\documentclass[ a4paper,10pt]{amsart}

\usepackage[foot]{amsaddr}

\usepackage[T1]{fontenc}

\usepackage[utf8]{inputenc}

\usepackage{lmodern}

\usepackage[english]{babel}

\usepackage{amsmath}
\makeatletter
\def\maketag@@@#1{\hbox{\m@th\normalfont\normalsize#1}}
\makeatother
\usepackage{amssymb}
\usepackage[alphabetic]{amsrefs}
\usepackage{mathtools}

\usepackage{mathrsfs}

\usepackage{xypic}

\usepackage{tikz-cd}

\usepackage{bookmark}

\usepackage{hyperref} 
\hypersetup{
	colorlinks=true,
	linkcolor=blue,
	filecolor=magenta,      
	urlcolor=cyan,
}

\usepackage{cleveref}

\usepackage{enumitem}
\usepackage{graphicx}
\usepackage{tikz}
\usepackage{tikz-cd}
\usetikzlibrary{matrix,arrows,decorations.pathmorphing}
\usepackage[all,cmtip]{xy}

\newtheorem*{thm-no-num}{Theorem}
\newtheorem*{cor-no-num}{Corollary}
\newtheorem{thm}{Theorem} [section]

\newtheorem{prop}[thm]{Proposition} 
\newtheorem{lm}[thm]{Lemma} 
\newtheorem{cor}[thm]{Corollary} 
\theoremstyle{definition}
\newtheorem*{question}{Question}
\newtheorem*{df-no-num}{Definition}
\newtheorem {df}[thm]{Definition}
\theoremstyle{remark} 
\newtheorem{rmk}[thm]{Remark}
\newtheorem{ex}[thm]{Example}

\newcommand{\PP}{\mathbb{P}}
\newcommand{\ZZ}{\mathbb{Z}}

\newcommand{\cl}[1]{\mathcal{#1}}
\newcommand{\sk}[1]{\mathscr{#1}}
\newcommand*{\sheafhom}{\mathcal{H}\kern -.5pt om}

\newcommand{\sA}{\sk{A}}

\newcommand{\sB}{\sk{B}}

\newcommand{\sC}{\sk{C}}

\newcommand{\sDP}{\sk{DP}}
\newcommand{\cE}{\mathcal{E}}
\newcommand{\sfE}{\mathsf{E}}
\newcommand{\sfH}{\mathsf{H}}

\newcommand{\sM}{\sk{M}}
\newcommand{\Ocal}{\cl{O}}

\newcommand{\Hcal}{\cl{H}}
\newcommand{\sH}{\sk{H}}
\renewcommand{\cL}{\cl{L}}
\newcommand{\Xcal}{\sk{X}}
\newcommand{\Ycal}{\cl{Y}}

\newcommand{\sS}{\sk{S}}
\newcommand{\sX}{\sk{X}}

\newcommand{\sU}{\sk{U}}

\newcommand{\Gm}{\mathbb{G}_{\rm m}}

\newcommand{\GLt}{\textnormal{GL}_3}

\newcommand{\Sp}{\mathrm{Sp}}
\newcommand{\Sm}{\mathrm{S}}
\newcommand{\bmu}{\mathbf{\mu}}
\renewcommand{\O}{\mathrm{O}}
\newcommand{\bfk}{\mathbf{k}}

\renewcommand{\H}{{\rm H}}
\newcommand{\K}{{\rm K}}
\newcommand{\M}{{\rm M}}

\newcommand{\Ar}{{\rm A}}

\newcommand{\Inv}{{\rm Inv}^{\bullet}}
\newcommand{\Spec}{{\rm Spec}}

\begin{document}
\title[Cohomological invariants of $\sM_{3,n}$ via level structures]
{Cohomological invariants of $\sM_{3,n}$ via level structures} 

\author{A. Di Lorenzo}
\email{andrea.dilorenzo@unipi.it}  
\address{Università di Pisa \\ Dipartimento di matematica \\ Largo Bruno Pontecorvo 5, 56127 \\ Pisa, Italy}
\begin{abstract}
We show that mod $2$ cohomological invariants of the moduli stack $\sM_{3,n}$ of smooth pointed curves of genus three contain a free module with generators in degree $0$, $2$, $3$, $4$ and $6$, formed by the invariants of the symplectic group $\Sp_6(2)$. We achieve this by showing that the torsor of full level two structures $\sM_{3,n}(2) \to \sM_{3,n}$ is versal. 

Along the way, we prove that the invariants of the stack of del Pezzo surfaces of degree two contain the invariants of the Weyl group $W(\sfE_7)$ and that the mod $2$ cohomology of $\sM_{3,n}$ is non-zero in degree three. Our main result holds also for the stack $\sA_3$ of principally polarized abelian threefolds.
\end{abstract}
\maketitle
\section*{Introduction}
Cohomological invariants were introduced by Serre \cite{GMS} as a tool for studying algebraic groups over a field, and their definition is inspired by characteristic classes in topology.

Pirisi extended this notion to the realm of algebraic stacks \cite{Pir}, generalizing both 
the cohomological invariants à la Serre and unramified cohomology à la Colliot-Th\'{e}lène--Ojanguren \cite{CTO}.

Moduli of curves are among the most well studied algebraic stacks, and their topological and algebraic invariants, e.g. cohomology groups and Chow ring, sometimes have beautiful hidden structures. One can ask the following.
\begin{question}
    What are the cohomological invariants of $\sM_{g,n}$, the moduli stack of $n$-pointed smooth curves of genus $g$? Is there any underlying structure?
\end{question}
Complete results have been obtained so far only for $\sM_{1,1}$ \cites{Pir, DLP-pos} and  $\sM_2$ \cites{Pir-even, DLP}. This paper focuses on the next case of interest, that is $g=3$. Denote $\Sp_6(2)$ the group of invertible symplectomorphisms of a $\ZZ/2$-vector space of rank six. Recall that the moduli stack $\sM_{3,n}(2)$ of curves with a full level two structure is an $\Sp_6(2)$-torsor over $\sM_{3,n}$, thus inducing a morphism $\sM_{3,n} \to \sB\Sp_6(2)$. Denote $\bfk$ the base field.
\begin{thm-no-num}[\Cref{thm:inj}, \Cref{cor:inv Sp6}]
    Assume that $\operatorname{char}(\bfk)\neq 2$ and let $\M$ be a cycle module annihilated by $2$.  Then
    \[\Inv(\sM_{3,n}, \M) \supset \M^{\bullet}(\bfk)\langle 1, \widetilde{w}_2, f_3,\widetilde{w}_4, \widetilde{w}_6 \rangle, \]
    where the free $\M^{\bullet}(\bfk)$-module on the right is formed by the invariants of $\sB\Sp_6(2)$, pulled back through the classifying morphism $\sM_{3,n} \to \sB \Sp_6(2)$. 
    The same statement holds for the moduli stack $\sA_3$ of principally polarized abelian threefolds.
\end{thm-no-num}
The invariants $\widetilde{w}_i$ come from the Galois-Stiefel-Whitney classes of $\sB \O_8$, the classifying stack of quadratic forms of rank eight (see \Cref{rmk:origin}), and the computation of $\Inv(\sB \Sp_6(2),\M)$ easily follows from previous works of Hirsch and Serre \cites{Hir, Ser}. As a simple application of our theorem, we deduce that mod $2$ cohomology of $\sM_{3,n}$ does not vanish in degree three, contrary to what happens for rational or $\ell$-adic cohomology. Suppose that $\bfk$ is algebraically closed. 
\begin{cor-no-num}[\Cref{thm:degree three}]
    We have $\dim \H_{\text{\'{e}t}}^3(\sM_{3,n},\ZZ/2)\geq 2$.
\end{cor-no-num}
A result that we prove along the way, which might be of some independent interest, is the following.
\begin{thm-no-num}[\Cref{cor:inj m31}]
    Let $\sDP_2$ be the stack of del Pezzo surfaces of degree two. Then the cohomological invariants of $\sDP_2$ contains the cohomological invariants of $\sB W(\sfE_7)$ as a submodule.
\end{thm-no-num}
\subsection*{Strategy of proof}
The strategy of proof for our main theorem is pretty simple: we show that $\sM_{3,n}(2) \to \sM_{3,n}$ is a versal $\Sp_6(2)$-torsor. 

In a nutshell: following a construction of Looijenga for the respective coarse moduli spaces \cite{Loo}, we consider the moduli stack $\sDP_{2,\rm ord}$ of degree two del Pezzo surfaces with an ordinary point. As every del Pezzo surface of degree two is a double cover of $\PP^2$ ramified over a smooth quartic, this stack maps to $\sM_{3,1}$. 

By decorating del Pezzo surfaces with a so called geometric marking and some more data, we construct a $W(\sfE_7)\times \bmu_2$-torsor over $\sDP_{2,\rm ord}$. The total space of this torsor is an open subset of a split torus, which fact is key for proving the versality, basically because twisted tori are unirational. We go back to $\sM_{3,1}(2)$ by leveraging well-known relations between geometric markings of del Pezzo surfaces and level two structures. 

The extension to $\sM_{3,n}$ follows by a simple argument that involves the embedding of $\sM_{3,1}$ in $\sM_{3,n}^{\rm ct}$, the stack of curves of compact type.

\subsection*{Relation with other works}
First, let us remark that, in hindsight, the same technique works for proving that $\Inv(\sM_{2,n})$ contains the cohomological invariants of $\sB S_6$, and that $\Inv(\sM_{1,n})$ contains the invariants of $\sB S_2$, thus giving another proof and extending previous results of Pirisi and the author \cites{Pir, DLP}.

Another source of cohomological invariants for $\sM_{g}$ that has been considered before is the $S_{2^g(2^{g-1}-1)}$-torsor associated to the \'{e}tale cover $\sS_g^- \to \sM_{g}$ given by odd spin curves; this induces a map $\sM_g \to \sB S_{2^g(2^{g-1}-1)}$ \cites{DLP2, JPP}, and one might wonder if the pullback of invariants along this map differ from the ones obtained via level structure. 

We show (see \Cref{prop:no new invs}) that this is not the case, i.e. there is a factorization $$\Inv(\sB S_{2^g(2^{g-1}-1)}) \longrightarrow \Inv(\sB\Sp_{2g}(2)) \longrightarrow \Inv(\sM_g).$$ 
For the $g=3$ case, we outline an approach that shows how the classes obtained via theta characteristics are related to the ones obtained via level structues. We also suggest that the same approach could provide a generalization of the main results of Jaramillo-Puentes and Pirisi, although the computations might easily become too hard to perform (see \Cref{rmk:rel hyp}). 
\subsection*{Structure of the paper}
The paper is divided into three sections. \Cref{sec:coh} covers the basics of the theory of cohomological invariants. In \Cref{sec:dp} we collect definitions and results involving del Pezzo surfaces, with a focus on the degree two case. \Cref{sec:curves} contains the proof of the main theorem and its applications.
\subsection*{Notation and assumptions}
We work over a base field $\bfk$ of characteristic not $2$, so every scheme and stack lives over $\Spec(\bfk)$, also when not explicitly mentioned.
    
    A prime denoted $\ell$ is always assumed to be prime with the characteristic of the base field. 
    
    We use the notation $\ZZ/n(i)=\mu_n^{\otimes i}$ for $i>0$, where $\mu_n$ is the group of roots of unity, and $\ZZ/n(-i)=\operatorname{Hom}(\ZZ/n(i),\Gm)$ for $i<0$. When dealing with these Galois modules, we always assume that $n$ is invertible in $\bfk$.
    
    We use the notation $\M^{\bullet}$ to indicate that an object is graded. In other terms, the notation $\M^{\bullet}$ stands for $\oplus_{n\in \ZZ} \M^n$, with $\M^n$ the graded pieces. The symbol $\M$ will always denote a torsion cycle module in the sense of Rost \cite{Ros}*{Definition 2.1}, whose torsion order we assume to be prime with the characteristic of the base fields.

    We denote $\Sp_{2g}(2)$ the group of invertible symplectomorphism of a $\ZZ/2$-vector space of rank $2g$.
\subsection*{Acknowledgments}
We warmly thank Roberto Pirisi and Angelo Vistoli, who were the first ones to suggest looking at level structures as a source of cohomological invariants for $\sM_{g,n}$. We also thank Davide Lombardo and Dan Petersen for fruitful exchanges.
\section{Cohomological invariants}\label{sec:coh}
\subsection{Basic definitions}
Let $\M=\M^{\bullet}$ be a cycle module in the sense of Rost \cite{Ros}*{Definition 2.1}. This is a functor on the category $({\rm Field}/\bfk)$ of field extensions.
\begin{ex}
    Examples of cycle modules to have in mind are:
    \begin{enumerate}
        \item Milnor K-theory, denoted $\K^{\rm Mil}$, and its mod $\ell$ version, denoted $\K^{\rm Mil}/\ell$;
        \item \'{E}tale cohomology $\H_{D}$ with coefficients in a torsion Galois module $D$ whose torsion orders do not divide char($\bfk$), e.g. $D=\ZZ/n(i)$. More precisely, we have $\H^{\bullet}_D(-)=\H^{\bullet}_{\text{ \'{e}t}}(-,D(\bullet))$.
    \end{enumerate}
\end{ex}
Recall that a DVR $(R,\nu)$ with fraction field $F$ and residue field $\bfk(\nu)$ is called \emph{geometric} when ${\rm tr.deg}_{\bfk}(F)={\rm tr.deg}_{\bfk}(\bfk(\nu))+1$, in which case for any cycle module $\M$ there is a well defined restriction map $\M^{\bullet}(F) \to \M^{\bullet}(\bfk(\nu))$.

Let $\Xcal$ be an algebraic stack over $\bfk$, and let $P_{\Xcal}$ be the restriction of the functor of points of $\Xcal$ over the category $({\rm Field}/\bfk)$.  
\begin{df}[\cite{Pir}*{Definition 2.2},\cite{DLP}*{Definition 2.3}]
    A \emph{cohomological invariant} $\alpha$ of $\Xcal$ with coefficients in $\M$ is a natural transformation $P_{\Xcal} \longrightarrow \M^{\bullet}$ satisfying the following condition:
    \begin{itemize}
        \item[($\star$)] for any geometric DVR $(R,\nu)$ and any object $\xi \in \Xcal(R)$, we have that $\alpha(\xi_F)|_{\bfk(\nu)}=\alpha(\xi_{\bfk(\nu)})$.
    \end{itemize}
    A cohomological invariant has \emph{cohomological degree} $d$ if it takes values in the graded piece $\M^{d}$. The graded ring of cohomological invariants is denoted $\Inv(\Xcal,\M)$ and sometimes shortened into $\Inv(\Xcal)$.
\end{df}
\begin{rmk}
 For $G$ a smooth algebraic group over $\bfk$, the condition $(\star)$ is always verified, and the cohomological invariants of the classifying stack $\sB G$ coincide with the cohomological invariants of the algebraic group $G$ defined by Serre \cite{GMS}*{Definition 1.1}.
\end{rmk}
\begin{rmk}\label{rmk:smNis}
For a smooth scheme $X$ over $\bfk$, let $\Hcal^d_{\ZZ/n(i)}$ be the sheafification (in the Zariski topology) of the presheaf
\[ U \longmapsto \H_{\text{\'{e}t}}^{d}(U,\ZZ/n(d+i)) \]
where $\ZZ/n(i)=\mu_n^{\otimes i}$ for $i\geq 0$, and $\ZZ/n(-1)=\ZZ/n$.
One way to introduce unramified cohomology of schemes over a field $\bfk$ with coefficients in $\ZZ/n(i)$ is then as follows: 
\[ \H_{nr}^{d}(X,\ZZ/n(i)):= \H^0_{\rm Zar}(X,\Hcal^d_{\ZZ/n(i)}).\]
The above definition should be changed in order to work for smooth algebraic stacks, because Zariski topology of stacks is too coarse. 

Assuming that the base field $\bfk$ is infinite and that the algebraic stack has affine stabilizers, the right topology turns out to be the smooth-Nisnevich one \cite{Pir}*{\S 3}: a smooth-Nisnevich morphism $f:\Xcal\to\Ycal$ is a representable morphism of algebraic stacks which is smooth and such that for every field $K$ every $K$ point of $\Ycal$ admits an isomorphic $K$-point which lifts to a $K$-point of $\Xcal$. For finite fields, one needs to use $m$-smooth-Nisnevich coverings, instead of ordinary ones.

Let $\Xcal$ be a smooth algebraic stack over $\bfk$ with affine stabilizers. Then one can define
\[ \H_{nr}^d(\Xcal,\ZZ/n(i)):=\H^0_{\text{sm-N}}(\Xcal,\Hcal^d_{\ZZ/n(i)}), \]
and we have $\Inv(\Xcal,\ZZ/n(i))=\H_{nr}^{\bullet}(\Xcal,\ZZ/n(i))$.
In this sense, the notion of cohomological invariants ties together the original definition of cohomological invariants of algebraic groups à la Serre, and unramified cohomology. Observe also that cohomological invariants form a sheaf in the smooth-Nisnevich topology.

Contrary to what happens for schemes, unramified cohomology groups of stacks do not form a birational invariant. A simple counterexample is given by the classifying stack $\sB\mu_2$, which has non-trivial cohomological invariants. 
\end{rmk}
\subsection{Versal torsors}
Versal torsors are, loosely speaking, good enough approximations of universal torsors.
\begin{df}
    A $G$-torsor $P\to \sX$ over an algebraic stack $\sX$ is \emph{versal} if for any infinite field $F\supset\bfk$, any $G$-torsor $Q\to\Spec(F)$ and any dense open substack $U\subset \sX$, there exists a morphism $f\colon \Spec(F) \to U$ such that $\Spec(F)\times_{\sX} P \simeq Q$.
\end{df} 
We care about versal torsors because of the following lemma, which is well known.
\begin{lm}\label{lm:versal implies injective}
    Let $\sX$ be an algebraic stack, and $f\colon \sX \to \sB G$ a morphism whose associated $G$-torsor $P\to \sX$ is versal. Then the pullback homomorphism $$f^*\colon \Inv(\sB G) \longrightarrow \Inv (\sX)$$ is injective. 
\end{lm}
\begin{proof}
    Suppose that $f^*\alpha=0$. We want to show that, for any field $F$ and any $G$-torsor $Q\to\Spec(F)$, we have $\alpha(Q\to\Spec(F))=0$. By definition of versality of $P\to\sX$, the classifying morphism $g\colon \Spec(F) \to \sB G$ factors through $h\colon\Spec(F)\to \sX$. We deduce that $\alpha(Q\to\Spec (F))=g^*\alpha=h^*(f^*\alpha)=0$. 
\end{proof}
\subsection{Relation with cohomology}
Given an element $\alpha$ in $\H_{\text{\'{e}t}}^d(\Xcal,\ZZ/n(i))$, we can define a cohomological invariant of degree $d$ with coefficient in $\H_{\ZZ/n(i)}$ as
\[ (\xi\colon\Spec(K)\to \Xcal) \longmapsto \xi^*\alpha.\]
In this way we obtain a homomorphism
\begin{equation}\label{eq:neg} \H_{\text{\'{e}t}}^\bullet(\Xcal,\ZZ/n(i)) \longrightarrow \Inv(\Xcal,\H_{\ZZ/n(i)}).\end{equation}
\begin{df}
    We say that a cohomological invariant \emph{comes from cohomology} if it belongs to the image of \eqref{eq:neg}.
\end{df}
\begin{rmk}
    The negligible cohomology of an algebraic group $G$ \cite{GMS}*{\S 26} is the kernel of \eqref{eq:neg} when $\Xcal=\sB G$. When $\bfk$ contains an $n^{\rm th}$ primitive root of the unity, the cycle classes of algebraic cycles mod $n$ are negligible.
\end{rmk}
Assume that $\bfk$ is algebraically closed, and let $\operatorname{CH}^d(-)$ denote the integral Chow group of degree $d$. For simplicity, we further assume that either
\begin{enumerate}
    \item $X$ is a quasi-projective scheme and the $G$-action is linearized, or
    \item $G$ is connected and $X$ is equivariantly embedded in a normal variety, or
    \item $G$ is special, i.e. every $G$-torsor is Zariski-locally trivial.
\end{enumerate}
These assumptions are made so that for $U$ a scheme endowed with a free $G$-action, the quotient $X\times U/G$ is a scheme \cite{EG}*{Proposition 23}; dropping the assumptions above would make us work with algebraic spaces, and the results used below are not stated for algebraic spaces.
\begin{prop}[\cite{Gui}*{Corollary 5.1.6}] \label{prop:exact seq 3}
    Let $X$ be a smooth scheme and $G$ an algebraic group satisfying the assumptions above, and set $\sX:=[X/G]$. Then the following is exact
    \[ \H_{\text{\'{e}t}}^3(\Xcal,\ZZ/n) \longrightarrow \operatorname{Inv}^3(\Xcal,\H_{\ZZ/n})\longrightarrow\operatorname{CH}^2(\Xcal)\otimes\ZZ/n\overset{{\rm cl}}{\longrightarrow} \H_{\text{\'{e}t}}^4(\Xcal,\ZZ/n), \]
    where the first morphism is \eqref{eq:neg} and the last morphism is the cycle class map.
\end{prop}
\begin{proof}
For a smooth scheme $X$, we have the Bloch-Ogus-Rost interpretation of the second page of the coniveau spectral sequence \cite{BO}*{Corollary 6.3}, \cite{Ros}*{Remark 5.3 and Corollary 6.5}
\[ E^{p,q}_2=\Ar^p(X,\H^q_{\ZZ/n}) \Longrightarrow \H_{\text{\'{e}t}}^{p+q}(X,\ZZ/n), \]
where by $\Ar^p(X,\M^q)$ we denote the $p^{\rm th}$ cohomology group of the complex 
\[ \oplus_{x\in X^{(0)}}\M^{q}(\bfk(x)) \longrightarrow \oplus_{x \in X^{(1)}} \M^{q-1}(\bfk(x))\longrightarrow \oplus_{x \in X^{(2)}} \M^{q-2}(\bfk(x)) \longrightarrow\cdots\] 
We also have that $\Ar^2(X,\H^2_{\ZZ/n})=\operatorname{CH}^2(X)\otimes\ZZ/n$ \cite{BO}*{Theorem 7.7}, \cite{Ros}*{Remark 5.1}. As the spectral sequence is zero outside of the first quadrant and above the diagonal, the differential $E^{0,3}_2 \longrightarrow E^{2,2}_2$ fits into the exact sequence
\[ \H_{\text{\'{e}t}}^{3}(X,\ZZ/n) \longrightarrow E^{0,3}_2 \longrightarrow E^{2,2}_2 \longrightarrow \H_{\text{\'{e}t}}^{4}(X,\ZZ/n) \]
where the first morphism concides with \eqref{eq:neg} and the last morphism is the cycle class map \cite{BO}*{(7.2.1)}. 

Now we use the standard argument of equivariant approximation à la Borel: given a quotient stack $\Xcal=[X/G]$ satisfying our assumptions, we can always find an open subset $U$ of a $G$-representation such that (1) $G$ acts freely on $U$ and (2) the complement of $U$ has sufficiently high codimension \cite{EG}*{Lemma 9}.  

We can therefore apply the argument above to the smooth scheme $X'=X\times U/G$, and the cohomology groups, cohomological invariants and Chow groups of $X'$ involved in the exact sequence, as well as the morphisms of the complex, coincide with the ones of $\Xcal$. 
\end{proof}
\begin{cor}\label{cor:hom triv}
    The cohomological invariants of degree three come from cohomology if and only if the subgroup $\operatorname{CH}^2(X)_{hom}$ of homologically trivial cycles (i.e., the kernel of the cycle class map) is trivial.
\end{cor}
\section{Del Pezzo surfaces}\label{sec:dp}
In this section we outline some basic facts about del Pezzo surfaces. Most of the statements are generalizations over general base schemes or reformulations in stacky terms of well known facts for del Pezzo surfaces over an algebraically closed field \cite{DO}*{Chapter V}.
\begin{df}
    A \emph{del Pezzo surface} of degree $d$ is a smooth, proper and integral surface whose anticanonical line bundle is ample and has degree $d$.

    A del Pezzo surface $X$ over a base $S$ is a proper and smooth morphism $X\to S$ whose geometric fibers are del Pezzo surfaces.
\end{df}
For $d>2$ the anticanonical line bundle is very ample; for $d=2$ it defines a double cover of $\PP^2$ ramified over a smooth quartic, and for $d=1$ the linear system $|-2K_X|$ realizes $X$ as a double cover over a quadric cone $Q$ in $\PP^3$, ramified along the smooth genus four curve obtained by intersecting $Q$ with a smooth cubic surface.

\subsection{Geometric markings}
Every del Pezzo surface of degree $d\neq 8$ over an algebraically closed field can be realized as the blow-up of $\PP^2$ at $9-d$ points, where this constraint comes from imposing $\deg(-K_X)=d$.

Furthermore, ampleness of $-K_X$ puts further constraints on this set of points; indeed, the induced rational morphism $\varphi\colon\PP^2 \dashrightarrow \PP^d$ is given by the linear system of cubics through the $9-d$ points: if for instance three of the distinguished points were to be collinear, the line through those points would be contracted by $\varphi$. Similar arguments show that no six of them can lie on a conic, and no eight of them can lie on a nodal cubic, with the node being one of the points \cite{DO}*{Chapter V, Remark 1}.
\begin{df}\label{df:geom real}
    For $X$ a del Pezzo surface of degree $d$ over a base scheme $S$, a \emph{geometric realization} (or a blowing-down structure \cite{DO}*{Chapter V, 1, pag. 63}) is a triple $(X_0\to S, p_1,\ldots, p_{9-d},\varphi)$ where
    \begin{enumerate}
        \item[(i)] $X_0 \to S$ is a Severi-Brauer variety of relative dimension two, i.e. a smooth $S$-scheme that \'{e}tale locally on $S$ is isomorphic to a relative $\PP^2$,
        \item[(ii)] the $p_i$'s are disjoint $S$-points such that, on every geometric fiber, no three of them are collinear, no six of them lie on a conic, and no nine of them lie on a nodal cubic, with the node being one of the points, and
        \item[(iii)] $\varphi$ is an isomorphism between $X$ and the blow-up $\widetilde{X}_0$ of $X_0$ at the $p_i$'s.
    \end{enumerate}
    Isomorphisms of geometrical realizations are isomorphisms of the del Pezzo surfaces that commute with the identification as blow-ups of Severi-Brauer varieties.
\end{df}
The Picard group of a del Pezzo surface $X$ over an algebraically closed field is therefore
\[\operatorname{Pic}(X) \simeq \ZZ\cdot H \oplus \ZZ\cdot E_1\oplus\cdots\oplus \ZZ\cdot E_{9-d},\]
and this lattice together with the intersection pairing forms a unimodular odd lattice of signature $(1,9-d)$. We use the symbol $\sf H_{9-d}$ to denote this lattice \cite{DO}*{Chapter V, 2}.
\begin{df}
    For $X$ a del Pezzo surface of degree $d$ over a base scheme $S$, a \emph{geometric marking} is an $S$-isomorphism of lattices $\psi\colon\underline{\operatorname{Pic}}_{X/S} \simeq \sfH_{9-d}\times S$ which maps $-K_{X/S}$ to $[3H-\sum E_i]\times S$.

    Isomorphisms of geometrically marked del Pezzo surfaces $(X\to S,\psi)\simeq (X'\to S,\psi')$ are $S$-isomorphisms $\varphi\colon X \to X'$ such that $\psi\circ\varphi^* = \psi'$.
\end{df}
We can form the algebraic stack $\sDP_d^{\rm m}$ of geometrically marked del Pezzo surfaces of degree $d$, and the algebraic stack $\sDP_d^{\rm r}$ of del Pezzo surfaces of degree $d$ with a geometric realization.
\begin{prop}\label{prop:geom mark vs real}
    For $d\geq 1$, there is an equivalence of stacks $\sDP_d^{\rm m}\simeq \sDP_d^{\rm r}$.
\end{prop}
\begin{proof}
    Let $(X_0\to S, p_1,\ldots, p_{9-d},\varphi)$ be a geometric realization of $X$, as in \Cref{df:geom real}. A Severi-Brauer variety with a marked point has a line bundle $\Hcal$ of vertical degree $1$, unique up to tensor product with line bundles coming from $S$. There is a unique isomorphism $\sfH_{9-d} \times S \to \underline{\operatorname{Pic}}_{\widetilde{X}_0/S}$ that maps the $E_i$ to the (equivalence classes of) the exceptional divisors at $p_i$, the class of $\Hcal$ to $H$ and $3H-\sum E_i$ to $-K_{\widetilde{X}_0/S}$. By composing this with $\varphi^*$, we get a geometrically marked del Pezzo surface.

    Let $\pi\colon X\to S$ be a geometrically marked del Pezzo surface, and call $\cL$ any representative in $\operatorname{Pic}(X)$ of the image of $E_{9-d}$ through the marking.
    \begin{lm}
        The sheaf $\pi_*\cL$ is locally free of rank one. Moreover, for $m$ big enough, the sheaf $\pi_*(\cL(-K_{X/S})^{\otimes m})$ is locally free and the morphism $$\pi^*\pi_*(\cL(-K_{X/S})^{\otimes m}) \longrightarrow \cL(-K_{X/S}^{\otimes m})$$ is surjective.
    \end{lm}
    \begin{proof}[Sketch of proof]
        If $S$ is the spectrum of an algebraically closed field, the first statement is equivalent to saying $h^0(X,\Ocal(E_{9-d}))=1$, which is known \cite{DO}*{Chapter V, proof of Proposition 8}; in this setting, also the other two statements are easy to check. The general case follows from these ones by applying the cohomology and base change theorem.
    \end{proof}
    Call $X_{8-d}$ the image of the induced $S$-morphism $f\colon X\to \PP(\pi_*(\cL(-K_{X/S})^{\otimes m})^{\vee})$. Then $X_{8-d}\to S$ is proper and smooth, with geometric fibers del Pezzo surfaces of degree $d'=d-1$, and with a section $p_{9-d}:S\to Y$ given by $f(E_{9-d})$. Furthermore, the geometric marking descends to a geometric marking $\sfH_{9-d'} \simeq \underline{\operatorname{Pic}}_{X_{8-d}/S}$.

    By repeating this argument, we eventually get a Severi-Brauer variety $X_0\to S$ with markings $p_1,\ldots, p_{9-d}$ satisfying the conditions of \Cref{df:geom real}. Let $\widetilde{X}_0$ be the blow-up of $X_0$ at the $p_i$'s. By the universal property of the blow-up, there exists a unique morphism $g\colon X \to \widetilde{X}_0$ over $X_0$.
    \begin{lm}
        The morphism $g\colon X \to \widetilde{X}_0$ is an isomorphism over $X_0$.
    \end{lm}
    \begin{proof}[Sketch of proof]
        Up to passing to the normalization of $S$, we can assume that $\widetilde{X}_0$ is normal. If we prove that $X\to \widetilde{X}_0$ is quasi-finite with $\widetilde{X}_0$ normal, we can conclude that it is an isomorphism by Zariski's main theorem. Uniqueness of normalization then implies the statement in the case of a non-normal base.
        
        Quasi-finiteness can be checked on the geometric fibers, hence is enough to prove the statement for $S$ the spectrum of an algebraically closed field. This blows down to proving that the induced morphisms from $E_i$ to the exceptional divisors of the blow-up are quasi-finite. If this were not the case, the induced morphism would be a contraction, and the image of $X$ would not be the whole blow-up, which is a contradiction as all the varieties involved are proper.
    \end{proof}
    Putting all together, we have produced a morphism $\sDP_d^{\rm m} \to \sDP_d^{\rm r}$. The two morphisms that we constructed are one the inverse of the other.
\end{proof}
\subsection{Root systems}
Given a del Pezzo surface $X\to S$ of degree $d$, the orthogonal group $\O(\sfH_{9-d})$ acts freely and transitively on the set of isomorphisms $\underline{\operatorname{Pic}}_{X/S} \simeq \sfH_{9-d}\times S$, but not on the set of the geometric markings of $X\to S$; the latter is a torsor for the action of the stabilizer of the element $-K_{{9-d}}:=3H-\sum E_i$.

Let $\sfH_{9-d}^0$ be the orthogonal complement of $K_{9-d}$, which is a negative-definite lattice of rank $9-d$. Consider the subset $\Phi\subset \sfH_{9-d}^0$ of elements $\alpha$ such that $(\alpha,\alpha)=(-2)$. The set $\Phi$ forms a root system, with reflections $\sigma_\beta(\alpha):=\alpha+\langle \alpha, \beta \rangle \beta$. A basis of $\Phi$ is given by $\alpha_i=E_i-E_{i+1}$ for $i< 9-d$ plus $\alpha_0=H - E_1 - E_2 - E_3$. It forms a basis for $\sfH_{9-d}^0$ as well.
\begin{lm}
    The morphism of algebraic stacks $\sDP_{d}^{\rm m} \longrightarrow \sDP_d$ is a torsor for the Weyl group $W(\Phi)$ of the root system $\Phi$.
\end{lm}
\begin{proof}
For $X \to S$ a del Pezzo surface, we denote $\underline{\operatorname{Pic}}^0_{X/S}$ the orthogonal complement of the class of the canonical divisor $K_X$. The Weyl group $W(\Phi)$ generated by reflections therefore acts freely and transitively on the set of isomorphisms $\underline{\operatorname{Pic}}^0_{X/S} \simeq \sfH_{9-d}^0\times S$, and moreover there is bijection between the set of such isomorphisms and the set of isomorphisms $\underline{\operatorname{Pic}}_{X/S} \simeq \sfH_{9-d}\times S$ that sends the class of $K_{X/S}$ to $K_{9-d}$. 
\end{proof}
\subsection{Degree two case}\label{sec:deg two dp}
Let $\sDP_2$ be the moduli stack of del Pezzo surfaces of degree $2$. Over an algebraically closed field, del Pezzo surfaces of degree two are obtained by blowing-up seven points in general position in the plane. 

On the other end, we can study the anticanonical linear system $|-K_X|$ by looking at the induced rational morphism $\PP^2\dashrightarrow \PP^2$: the latter is given by mapping a point $p$ to the pencil of cubics that passes through the seven points and $p$ itself. Any such pencil has nine base points, hence there exists only a unique other point $\iota(p)$ in the plane defining the same exact pencil. Therefore, the rational morphism $\PP^2\dashrightarrow \PP^2$ has degree two, and the Geiser involution $\iota\colon \PP^2\dashrightarrow \PP^2$ is the one that swaps the points in the fibers.

The induced morphism $X\to \PP^2$ is thus a double cover, and its ramification divisor is a smooth quartic curve in the plane.
\begin{prop}
    Let $V$ be the $\GLt$-representation of quartic trinary forms, with action $A\cdot f(x_0,x_1,x_2):=\det(A)^2 f(A^{-1}(x_0,x_1,x_2))$, and let $D_V$ be the divisor of singular forms. Then there is an isomorphism $\sDP_2 \simeq [(V \smallsetminus D_V)/\GLt]$.
\end{prop}
\begin{proof}
    We can regard $\sDP_2$ as the stack whose objects over a scheme $S$ are double covers $X\to P$ with $P\to S$ a Severi-Brauer variety of relative dimension two, i.e. a smooth and proper morphism whose geometric fibers are isomorphic to $\PP^2$, and with the ramification divisor given by a family of smooth quartics in $P$. Stacks of cyclic covers are isomorphic to the claimed quotient stack \cite{AV}*{Corollary 4.2}.
\end{proof}
\begin{prop}
    The stack $\sDP_2$ is isomorphic to the order two root stack of $\sM_3\smallsetminus\sH_3$ along dual of the Hodge line bundle, i.e. the stack whose objects are smooth non-hyperelliptic curves $\pi\colon C\to S$ of genus three plus a line bundle $\cL$ and an isomorphism $\cL^{\otimes 2} \simeq (\pi_*\omega_{C/S})^\vee$.
\end{prop}
\begin{proof}
    Let $W$ be the $\GLt$-representation of quartic trinary forms, with action $A\cdot f(x_0,x_1,x_2) = \det(A)f(A^{-1}x_0,x_1,x_2)$ and let $D_W$ be the divisor of singular quartics, so that we have $\sM_3\smallsetminus\sH_3\simeq [(W\smallsetminus D_W)/\GLt]$.

    Consider the homomorphism $\GLt\to\GLt$ given by $A\mapsto\det(A)^{-1}A$, whose kernel is $\mu_2$. The induced morphism $f\colon\sB\GLt\to\sB\GLt$ is a $\mu_2$-gerbe, and we have a cartesian diagram
    \[\begin{tikzcd} \sDP_2\simeq [(V\smallsetminus D_V) /\GLt] \ar[r] \ar[d] & \sB\GLt \ar[d, "f"] \\
    \sM_3\smallsetminus\sH_3 \simeq [(W\smallsetminus D_W)/\GLt] \ar[r] & \sB\GLt.
    \end{tikzcd}\]
    
    To conclude the proof, we show that $f\colon \sB\GLt \to \sB\GLt$ is isomorphic to the order two root stack $\sqrt{\sB\GLt} \to \sB\GLt$ made along the dual of the determinant line bundle, i.e. the invertible sheaf on $\sB \GLt$ whose value at a vector bundle $\cE$ is $\det(\cE)^{-1}$.

    By definition the objects of $\sqrt{\sB\GLt}$ are triples $(\cE,\cL,\alpha\colon \cL^{\otimes (-2)}\simeq \det(\cE))$ with $\cE$ (respectively $\cL$) a rank three (respectively rank one) vector bundle over a scheme $S$. The morphism
    \[ \sB\GLt \longrightarrow\sqrt{\sB\GLt},\quad \cE\longmapsto (\det(\cE)^{-1}\cE,\det(\cE)) \]
    has an inverse given by
    \[ \sqrt{\sB\GLt} \longrightarrow \sB\GLt,\quad (\cE,\cL)\longmapsto \cE\otimes\cL, \]
    and the composition $\sB\GLt \to \sqrt{\sB\GLt}\to \sB\GLt$ is indeed $f$.
\end{proof}
\subsection{Pointed del Pezzo surfaces}
In this section we formulate for moduli stacks of del Pezzo surfaces and related objects results that are known for the respective coarse moduli spaces \cite{Loo}*{1.6}.

Let $\sC\to\sM_{3}\smallsetminus\sH_3$ be the universal curve, whose objects are pairs $(\pi\colon C\to S, p\colon S\to C \text{ section of }\pi)$. We call an $S$-point $p$ an \emph{ordinary} point if the tangent line to $p$ in $\PP(\pi_*\omega_{C/S})$ intersects $C$ in an \'{e}tale divisor of degree two, i.e. for every geometric fiber the intersection locus is made of $p$ and two other distinct points.

Let $\sC_{\rm ord}$ be the open substack where the marking is an ordinary point, and set $\sDP_{2,\rm ord}:=\sDP_2 \times_{\sM_3} \sC_{\rm ord}$, which is therefore a $\mu_2$-gerbe over $\sC_{\rm ord}$. Finally, set $\sDP_{2,1}:=\sDP_2 \times_{\sM_3} \sM_{3,1}$.

Define the algebraic split torus
$T=\operatorname{Hom}_{\rm grp}(\Phi_{\sfE_7}\otimes \ZZ,\Gm)$ and let $D_T \subset T$ be the divisor of homomorphisms whose kernel contains an element of the root system. The torus $T$ is endowed with a natural $W(\sfE_7)\times\bmu_2$-action, namely
$((\sigma, \epsilon) \cdot \varphi)(x):= \varphi(\sigma\cdot x)^\epsilon $. 
\begin{df}
     Define the auxiliary stack $\sDP_{2, \rm ord}^{\rm aux}$ to be the stack over $\sDP_{2,\rm ord}$ whose objects are geometrically marked del Pezzo surfaces of degree two with a marked point on the branching divisor, an anticanonical divisor $D$ with a singular point at $p$, and an isomorphism $\alpha\colon \underline{\operatorname{Pic}}_{D/S}^0 \simeq{\Gm}\times S$.
\end{df}
\begin{lm}\label{lm:aux is mu2 torsor}
    The forgetful morphism $\sDP_{2,\rm ord}^{\rm aux} \longrightarrow \sDP_{2,\rm ord}^{\rm m}$ is a $\bmu_2$-torsor.
\end{lm}
\begin{proof}
    Let $X\to S$ be a geometrically marked del Pezzo surface of degree $2$, with $p$ an ordinary $S$-point of the branching divisor, and let $\pi\colon X \to P$ be the double cover with ramification divisor the smooth quartic curve $C$. By definition of ordinary point, \'{e}tale locally on $S$, the unique line $\ell\subset P$ tangent to $C$ at $\pi(p)$ intersects the quartic curve in two other points, say $\pi(q)$ and $\pi(q')$, with $q \neq q'$.
    
    There is a unique anticanonical divisor $D\subset X$ which is an irreducible curve having a node at $p$: it is the preimage along $\pi$ of $\ell$. Observe that $D$ is naturally endowed with an involution, inherited from $X$, and with the degree two divisor $q+q'$, where both $q$ and $q'$ are fixed by the involution, as they lie on the branching divisor as well.

    The morphism $\sDP_{2,\rm ord}^{\rm aux} \to \sDP_{2,\rm ord}^{\rm m}$ has a natural structure of $\bmu_2$-torsor over $\sDP^{\rm m}_{2,\rm ord}$ via $\alpha \longmapsto (-)^{\epsilon} \circ \alpha$.
\end{proof}
The following proposition extends to stacks a known result for the coarse moduli spaces \cite{Loo}*{Poposition 1.8}.
\begin{prop}\label{prop:torus parametrization}
   There is a $W(\sfE_7)\times \bmu_2$-equivariant isomorphism $\sDP_{2,\rm ord}^{\rm aux} \simeq T\smallsetminus D_T$. 
\end{prop}
We break down the proof in several parts. 
\begin{lm}\label{lm:from aux to torus}
    There is a $W(\sfE_7)\times \bmu_2$-equivariant morphism $\sDP_{2,\rm ord}^{\rm aux} \longrightarrow T\smallsetminus D_T$.
\end{lm}
\begin{proof}
    The geometric marking and the isomorphism $\alpha$ determines a unique isomorphism $\operatorname{Hom}_{S-{\rm grp}}(\underline{\operatorname{Pic}}^0_{X/S},\underline{\operatorname{Pic}}^0_{D/S}) \simeq T\times S$, hence the restriction homomorphism $\underline{\operatorname{Pic}}^0_{X/S} \to \underline{\operatorname{Pic}}^0_{D/S}$ that maps $\cL \to \cL|_D$ determines an element $\chi$ of $T\times S$. 

    As $X$ is geometrically marked, we can write the elements of the root system $\Phi_{\sfE_7}\times S$ inside $\underline{\operatorname{Pic}}_{X/S}$ as $Z_{ij}=E_i - E_j$ for $i\neq j$, $Z_{ijk}=H-E_i-E_j-E_k$ for $i,j,k$ distinct, $Z_{i}=2H-E_{i_1}-\ldots -E_{i_6}$, where by \Cref{prop:geom mark vs real} the class of $H$ (respectively $E_i)$ comes from the hyperplane class of a Severi-Brauer variety $X_0 \to S$ (respectively is the exceptional divisor obtained by blowing up a point $x_i$ in $X_0$). Most importantly, the points $x_1,\ldots, x_7$ are in general position, i.e. they are distinct, no three of them lie on a line, no six of them lie on a conic.
    
    Observe that on every geometric fiber $(E_i\cdot D)=1$, hence $Z_{ij}|_D\sim 0$ if and only if $E_i \cap D = E_j \cap D$, which would imply $x_i=x_j$. Therefore, we have $Z_{ij}|_D \not\sim 0$. Similarly, we have that if $Z_{ijk}|_D \sim 0$ then there exists a line in $X_0$ passing through $x_i$, $x_j$ and $x_k$, which again cannot happen. For the same reason, we have $Z_{i}|_D \not\sim 0$ as otherwise we would have that six of the seven points lie on a conic. We have proved that $\chi$ lies in $(T\smallsetminus D_T )\times S$.
\end{proof}
\begin{lm}\label{lm:from torus to aux}
    There is a $W(\sfE_7)\times \bmu_2$-equivariant morphism $T\smallsetminus D_T \longrightarrow \sDP_{2,\rm ord}^{\rm aux}$.
\end{lm}
\begin{proof}
    We can regard $T\smallsetminus D_T$ as the stack whose objects consist of the following data:
    \begin{enumerate}
        \item A nodal, pointed, irreducible curve of genus one $(C,x_1)\to S$,
        \item a collection of seven non trivial line bundles $\chi_1,\ldots, \chi_7$ of degree zero, and
        \item an isomorphism $\underline{\operatorname{Pic}}^0_{C/S}\simeq \Gm\times S$.
    \end{enumerate}
    Indeed, the unique automorphism of a nodal pointed irreducible curve of genus one (i.e., the involution) acts non-trivially on the relative Jacobian of the curve, thus by fixing the isomorphism $\underline{\operatorname{Pic}}^0_{C/S}\simeq \Gm\times S$ we force the family of curves $(C,x_1)\to S$ to be trivial, not just isotrivial.

    Given these data, we can produce other six points $x_2,\ldots, x_7$ on $C\to S$ by imposing $\Ocal(x_i-x_{i+1})\simeq \chi_{i}$. Furthermore, we have a line bundle $\cL$ of degree three given by $\chi_7\otimes\Ocal(x_1+x_2+x_3)$, which we can use to embed $C$ in $X_0=\PP\H^0(C,\cL)\times S$.

    We can blow-up $X_0$ at $x_1,\ldots, x_7$: observe that by construction these seven points are in general position, hence the resulting surface is a del Pezzo surface $X\to S$ of degree two. By \Cref{prop:geom mark vs real} this surface inherits a geometric marking. 

    Consider the strict transform $D$ of $C$, together with the preimage $p$ of the node $x_0$ of $C$. To conclude the proof, we only need to show that $p$ is an ordinary point. This amounts to showing that it belongs to the branching divisor, and that the image of $D$ through the double cover $X\to \PP^2_S$ is a line tangent to the image of $p$ and intersecting the ramification divisor in other two distinct points.
    
    All these properties follows from the node $x_0$ being fixed by the Geiser involution of $X_0$, as then the preimage $p$ in $X$ of $x_0$ must lie on the branching locus of the ramified double cover $X \to \PP^2\times S$; moreover, such point must be an ordinary one, as otherwise $C$ would either have a cusp, a tacnode or be reducible.  
    
    To prove that $x_0$ is fixed, observe that any other cubic in the pencil of cubics passing through $x_1,\ldots, x_7$ intersects $C$ in nine points counted with multiplicity, hence these intersection points can only be $x_1,\ldots, x_7$ and $x_0$, as the latter is counted with multiplicity two, being a node. By definition of Geiser involution, we have $\iota(x_0)=x_0$.  We have thus defined a morphism $T\smallsetminus D_T \longrightarrow \sDP_{2,\rm ord}^{\rm aux}$, which is easy to check to be equivariant with respect to the $W(\sfE_7)\times\bmu_2$-action.
\end{proof}
\begin{proof}[Proof of \Cref{prop:torus parametrization}]
    In \Cref{lm:from aux to torus} and \Cref{lm:from torus to aux} we constructed two $W(\sfE_7)\times\bmu_2$-equivariant morphisms $\sDP_{2,\rm ord}^{\rm aux} \to T\smallsetminus D_T$ and $T\smallsetminus D_T \to \sDP_{2,\rm ord}^{\rm aux} $, which are one the inverse of the other. 
\end{proof}
\begin{cor}\label{cor:comm diagram aux}
    The induced diagram of algebraic stacks
    \[
    \begin{tikzcd}
        \sDP_{2,\rm ord} \ar[r] \ar[d] & \sB (W(\sfE_7)\times \bmu_2) \ar[d] \\
         \sDP_{2} \ar[r]& \sB W(\sfE_7),
    \end{tikzcd}
    \]
    where the top arrow is given by the torsor $\sDP_{2,\rm ord}^{\rm aux}$ and the other by $\sDP_{2}^{\rm m}$, is commutative.
\end{cor}
\section{Moduli of curves of genus three}\label{sec:curves}
Recall that a full level structure of order two on a smooth curve of genus $g$ is an isomorphism of the symplectic $\ZZ/2$-vector space 
\[\H_{\text{\'{e}t}}^1(C,\ZZ/2), \quad \langle\cdot,\cdot\rangle\colon \H^1_{\text{\'{e}t}}(C,\ZZ/2)\times\H^1_{\text{\'{e}t}}(C,\ZZ/2) \to \H^2_{\text{\'{e}t}}(C,\ZZ/2)\simeq\ZZ/2,\]
with the symplectic vector space $\ZZ/2^{\oplus 2g}$ endowed with the trivial symplectic form $\Omega$, i.e. is the choice of a symplectic basis of $\H^1_{\text{\'{e}t}}(C,\ZZ/2)$. 

Equivalently, one can think of $\H^1_{\text{\'{e}t}}(C,\ZZ/2)$ as the finite subscheme of the $2$-torsion points of the Jacobian variety of $C$, with symplectic structure given by the Weil pairing. 
The definition of level structures therefore generalizes to principally polarized abelian varieties as well: given an abelian variety $A$ over $\bfk$ of dimension $g$, the set of $2$-torsion points of $A$ can be regarded as a $\ZZ/2$-vector space of dimension $g$, with symplectic structure given by the Weil pairing. 

Let $\sM_{g}(2)\to\sM_g$ be the moduli stack of smooth curves of genus three with level structure of order two, i.e.
\[ \sM_g(2)(S)= \left\{ C \to S\text{, } \alpha\colon(\H^1_{\text{\'{e}t}}(C,\ZZ/2),\langle\cdot,\cdot\rangle)\simeq (\ZZ/2^{\oplus 6},\Omega) \right\}.\]
Let $\Sp_{2g}(2)$ denote the finite group of invertible symplectomorphisms of $(\ZZ/2^{\oplus 6},\Omega)$. By construction $\sM_g(2)$ has a natural structure of $\Sp_{2g}(2)$-torsor over $\sM_g$, thus inducing the classifying morphism $\sM_g(2) \to \sB \Sp_{2g}(2)$. 

Observe that the morphism above can be extended to $\sM_g^{\rm ct}$, the moduli stack of stable curves of compact type, and it factors via the Torelli morphism through $\sA_g$, the moduli stack of principally polarized abelian varieties of dimension $g$, where $\sA_g \to \sB\Sp_{2g}(2)$ corresponds to the torsor $\sA_g(2)\to \sA_g$ given by the stack of principally polarized abelian varieties with a level structure of order two.


We now specialize to the case $g=3$. Observe that there is a factorization
\begin{equation}\label{eq:composition} \sM_{3,n} \longrightarrow \sM_{3,n}^{\rm ct} \longrightarrow \sA_3 \longrightarrow \sB\Sp_6(2).\end{equation}
We aim at proving the following.
\begin{thm}\label{thm:inj}
    Assume that $\operatorname{char}(\bfk)$ is not $2$. For $\sX$ any of the stacks appearing in \eqref{eq:composition}, the pullback homomorphism $\Inv(\sB\Sp_6(2),\M) \hookrightarrow\Inv(\sX,\M)$ is injective.
\end{thm}
\subsection{Symplectic spaces over $\ZZ/2$} \label{sub:symp}
We recall some basic definitions and constructions for symplectic spaces over $\ZZ/2$ \cite{Mum}, \cite{HG}*{Sections 1-2}.
Let $(V,\langle\cdot,\cdot\rangle )$ be such a space, of dimension $2g$.
\begin{df}
    A quadratic form $q$ over a symplectic $\ZZ/2$-space $(V,\langle \cdot,\cdot\rangle)$ is a function $V\to\ZZ/2$ such that $q(x+y)+q(x)+q(y)=\langle x,y\rangle$.
\end{df}
We denote the set of quadratic forms of $V$ as $QV$: regarding $V$ as a group via its additive structure, we see that $QV$ forms a $V$-torsor, with the $V$-action defined via the formula $(q+v)(x)=q(x)+\langle v,x\rangle$. If $V=X_1\oplus X_2$ is an isotropic decomposition, i.e. $\langle \cdot,\cdot\rangle|_{X_i}=0$ for $i=1$, $2$, we can define a quadratic form $q(x+y)=\langle x,y\rangle$. 
The natural action of $\Sp(V)$ on $QV$ has two orbits, one of which is made of those quadratic forms stemming from isotropic decompositions.
\begin{df}
    The Arf invariant $a(q)$ of a quadratic form is $0$ or $1$, depending on whether $q$ belongs to one orbit or the other or, equivalently, whether $q$ stems from an isotropic decomposition or whether $q$ has $2^{g-1}(2^g+1)$ zeros or not. 
\end{df}
The set $W=V\cup QV$ inherits a natural structure of $\ZZ/2$-vector space thanks to the $V$-action on $QV$. 
\begin{df}
    An Aronhold basis for $W$ is a basis $q_1,\ldots, q_{2g+1}$ of $W$ with $q_i\in QV$ and such that, for any quadratic form $q$, the number of elements of the basis necessary for writing down $q$ is equal to $a(q)$ modulo $4$.
\end{df}
Aronhold basis exist with $a(q_i)=0$ for $g\equiv 0,1$ mod $4$, and with $a(q_i)=1$ for $g\equiv 2,3$ mod $4$. Moreover, the $\Sp(V)$-action on the set of Aronhold basis is free and transitive, thus making it into an $\Sp(V)$-torsor.
\begin{lm}\label{lm:aronhold to basis}
    Let $V$ be a symplectic space over $\ZZ/2$ of dimension $2g$. The choice of a symplectic basis determines an Aronhold basis, and viceversa.
\end{lm}
\begin{proof}
    Given an Aronhold basis we set
    \begin{align*}
        &e_1=q_1+q_2,  &&f_1=q_1+q_{2g+1}, \\
        &e_2=q_3+q_4,  &&f_2=q_1+q_2+q_3+q_{2g+1}, \\
        &\vdots  &&\vdots \\
        &e_g=q_{2g-1}+q_{2g}, &&f_g=q_1+\ldots + q_{2g-1}+q_{2g+1}.
    \end{align*}
    which forms a symplectic basis of $V$. This morphism of set is $\Sp(V)$-equivariant, hence an isomorphism.
\end{proof}
The theory outlined so far establishes a connection between level structures of order $2$ on curves and theta characteristics \cite{Mum}.
\begin{lm}\label{lm:theta to q}
    For $C$ a smooth curve, let $V$ be the $\ZZ/2$-vector space $\operatorname{Jac}(C)[2]$, endowed with the symplectic form given by the Weil pairing. Then there is a natural isomorphism between $QV$ and the set of theta characteristics of $C$; moreover, the set of even theta characteristics (respectively odd theta characteristics) corresponds to the forms whose Arf invariant is $0$ (respectively $1$).
\end{lm}
\begin{proof}
    Given a theta characteristic $\theta$, we define the function
    \[ q_\theta\colon\cL \longmapsto h^0(\theta)+h^0(\cL\otimes \theta) \text{ mod } 2.\]
    This is a quadratic form with respect to the Weil pairing \cite{Mum}*{Page 182, Identity ($\star$)}, and hence it defines a morphism from the set of theta characteristics to $QV$. As this morphism is equivariant with respect to the $V$-action on both sets, we deduce that it is an isomorphism.
\end{proof}
In particular, the combination of \Cref{lm:aronhold to basis} and \Cref{lm:theta to q} implies the following.
\begin{lm}\label{lm:level structures to aronhold}
    There is an equivalence of $\Sp_{2g}(2)$-torsors between the set of full level $2$ structures on $C$ and the set of Aronhold basis of theta characteristics.
\end{lm}
In the case of non-hyperelliptic smooth curves of genus three, whose canonical model is a smooth quartic plane curve, Aronhold basis corresponds to sets of seven bitangents such that no three of them intersects the quartic curve in the same locus of intersection with a conic.
\subsection{Geometric markings and level structures}
In what follows, we adopt the same notation used in \Cref{sec:deg two dp}. Let $X$ be a del Pezzo surface of degree two: recall that the Geiser involution $\iota\colon X \to X$ acts as $-1$ on $\operatorname{Pic}^0(X)$, the orthogonal complement of $K_X$, and that the root system generating this lattice is of type $\sfE_7$. The involution induces a splitting $W(\sfE_7)\simeq \Sp_6(2) \times \{ \pm 1 \}$. 
In particular, the involution $-1$ acts on $\sDP_2^{\rm m}$, the stack of geometrically marked del Pezzo surfaces, thus the quotient $\sDP_2^{\rm m}/\{ \pm 1 \}$ inherits a structure of $\Sp_6(2)$-torsor over $\sDP_2$. 
\begin{prop}\label{prop:iso DPm and M3(2)}
    There is an $\Sp_6(2)$-equivariant isomorphism $\sDP_2^{\rm m}/\{ \pm 1\} \simeq \sDP_2\times_{\sM_3} \sM_3(2)$ of $\Sp_6(2)$-torsors over $\sDP_2$.
\end{prop}
\begin{proof}
    Let $X\to S$ be a del Pezzo surface, with $C\subset X$ the branching divisor. A geometric marking of $X$ determines seven families of exceptional divisors $E_1,\ldots, E_7$. By taking the image of $E_i$ through the double cover $X\to P$, we obtain seven families of bitangents of $C$, the ramification divisor of the double cover. These families form an Aronhold basis, hence by \Cref{lm:level structures to aronhold} they determine a full level $2$ structure on $C$. This defines a morphism $\sDP_2^{\rm m} \to \sDP_2\times_{\sM_3} \sM_{3}(2)$.
    
    If instead of taking the $E_i$'s we were to take the divisors $\iota^*E_i$'s conjugated to the $E_i$'s via the involution, the resulting level $2$ structure on $C$ would be the same. This shows that the morphism above factors through $\sDP_2^{\rm m}/\{\pm 1\}$. 
\end{proof}
\begin{cor}\label{cor:comm diag dp and m3}
    The diagrams of algebraic stacks
    \[
    \begin{tikzcd}
        \sDP_{2,1} \ar[r,"\sDP_{2,1}^{\rm m}"] \ar[d] & \sB W(\sfE_7) \ar[d] \\
        \sM_{3,1} \ar[r, "\sM_{3,1}(2)"] & \sB \Sp_6(2)
    \end{tikzcd}, \quad
    \begin{tikzcd}
        \sDP_2 \ar[r,"\sDP_2^{\rm m}"] \ar[d] & \sB W(\sfE_7) \ar[d] \\
        \sM_3 \ar[r, "\sM_3(2)"] & \sB \Sp_6(2)
    \end{tikzcd}
    \]
    are commutative.
\end{cor}
\begin{prop}
    The $W(\sfE_7)\times\bmu_2$-torsor $\sDP_{2, \rm ord}^{\rm aux} \to \sDP_{2,\rm ord}$ is versal.
\end{prop}
\begin{proof}
    Set $G=W(\sfE_7)\times\bmu_2$. We proved in \Cref{prop:torus parametrization} that $\sDP_{2, \rm ord}^{\rm aux}$ is isomorphic to an open subset of a split torus $T$ of rank $7$. Let $Q \to \Spec(K)$ be a $G$-torsor: then our claim amounts to showing that the classifying morphism $\Spec(K) \to \sB G$ lifts to a morphism $\Spec(K) \to [T/G]$. Equivalently, we have to produce a rational section of the twisted $K$-torus $(T\times Q)/G$. This follows from every twisted torus being unirational.
\end{proof}
\begin{cor}\label{cor:inj m31}
    We have that
    \begin{enumerate}
        \item the morphism $\sDP_{2,\rm ord} \to \sB (W(\sfE_7)\times \bmu_2)$ induces an injective homomorphism $\Inv(\sB (W(\sfE_7)\times\bmu_2)) \to \Inv(\sDP_{2,\rm ord})$,
        \item the morphism $\sDP_{2} \to \sB W(\sfE_7)$ induces an injective pullback homomorphism $\Inv(\sB W(\sfE_7)) \to \Inv(\sDP_2)$,
        \item the morphism $\sM_{3,1} \to \sB \Sp_6(2)$ induces an injective pullback homomorphism $\Inv(\sB\Sp_6(2)) \to \Inv(\sM_{3,1})$.
    \end{enumerate}
\end{cor}
\begin{proof}
    First observe that, for $G$, $H$ algebraic groups, the pullback homomorphism $\Inv(\sB H) \to \Inv(\sB (G\times H))$ is injective.
    Now, point (1) is a straightforward application of \Cref{lm:versal implies injective}. 
    
    The pullback homomorphism $\Inv(\sB W(\sfE_7)) \to \Inv(\sB (W(\sfE_7)\times \bmu_2))$ is injective, hence the composition $\Inv(\sB W(\sfE_7)) \to \Inv (\sDP_{2,\rm ord})$ is injective. Applying \Cref{cor:comm diagram aux} we deduce point (2). 
    
    Similarly, point (3) follows by applying \Cref{cor:comm diag dp and m3}, because $W(\sfE_7)$ is isomorphic to $ \Sp_6(2)\times\ZZ/2$, thus the pullback homomorphism $\Inv(\sB \Sp_6(2)) \to \Inv(\sB W(\sfE_7))$ is injective.
\end{proof}
\begin{rmk}
    As a matter of facts, we have proved that the $\Sp_6(2)$-torsor $\sM_3(2)\to\sM_3$ and its pullback to the universal curve $\sM_{3,1} \to \sM_3$ are versal.
\end{rmk}
\begin{proof}[Proof of \Cref{thm:inj}]
    
    The case $\sX=\sM_{3,1}$ has already been proved in \Cref{cor:inj m31}. Thanks to the factorization \eqref{eq:composition}, this also proves the cases $\sX=\sM_3$, $\sM_3^{\rm ct}$ and $\sA_3$. 

    Let $n\geq 2$ and pick any $(n+1)$-marked smooth rational curve $(D,p_1,\ldots,p_{n+1})$. Given a pointed smooth curve $(C,q)$ of genus three, we can define an $n$-marked smooth curve of genus three $C'$ by gluing $C$ and $D$ at the points $q$ and $p_{n+1}$; this defines an embedding $\sM_{3,1} \hookrightarrow \sM_{3,n}^{\rm ct}$ and the composition with the forgetful morphism $\sM_{3,n}^{\rm ct} \to \sM_3$ is equal to the universal curve $\sM_{3,1}\to\sM_3$. The injectivity of the pullback along $\sM_{3,1} \to \sB\Sp_6(2)$ then implies the remaining cases.
\end{proof}
\begin{rmk}
    The same arguments used for proving \Cref{thm:inj} works in genus two and one. The injectivity of the pullback homomorphisms along $\sM_{1,1} \to \sB \Sp_2(2)\simeq \sB S_2$ and $\sM_{2} \to \sB \Sp_4(2)\simeq \sB S_4$ has already been proved in \cite{Pir} and \cite{DLP}. To extend the result to the pointed cases, the only non trivial case to prove is the injectivity along of the pullback to $\sM_{2,1}$. 
    
    This follows by showing that the $S_6$-torsor $\sM_{2,1}(2)\to\sM_{2,1}$ is versal, which can be done by relating this torsor with the $S_6$-torsor $\sM_{0,6}\to \sM_{0.6}/S_6$, whose twisted forms are unirational.
\end{rmk}
\subsection{Cohomological invariants of $\Sp_6(2)$}
If we ignore what the cohomological invariants of $\sB\Sp_6(2)$ are, \Cref{thm:inj} does not tell us much; luckily, the invariants of $\sB\Sp_6(2)$ are easy to deduce leveraging the isomorphism of $\Sp_6(2)\times\ZZ/2 $ with $ W(\sfE_7)$, the Weyl group of the root system $\sfE_7$. 

The computation of the cohomological invariants of Weyl groups of root systems has been achieved independently by Hirsch and Serre. 
\begin{prop}[\cite{Hir}*{8.2}, \cite{Ser}*{Theorem A and Example 3}]
    Assume that $\operatorname{char}(\bfk)\neq 2$ and that the cycle module $\M$ is annihilated by $2$. Then $\Inv(W(\sfE_7))$ is a free $\M^{\bullet}(\bfk)$-module, with a basis given by
    \[ \Inv(\sB W(\sfE_7), \M)\simeq \M^{\bullet}(\bfk)\langle 1, \widetilde{w}_1, \widetilde{w}_2, \widetilde{w}_3, f_3, \widetilde{w}_1 f_3, \widetilde{w}_4, \widetilde{w}_5, \widetilde{w}_6, \widetilde{w}_7 \rangle \]
\end{prop}
\begin{rmk}\label{rmk:origin}
    There are several ways to define the generators above. All but two of these generators are pullbacks of Galois-Stiefel-Whitney invariants of $\sB\O_8$: recall that the Weyl group $W(\sfE_7)$ can be embedded in $W(\sfE_8)$, and that the latter acts faithfully via isometries on a vector space of rank eight. The induced embedding $W(\sfE_7)\subset \O_8$ induces a morphism of stacks $\sB W(\sfE_7) \to \sB\O_8$, along which we can pull back the Galois-Stiefel-Whitney invariants $\widetilde{w}_1,\ldots,\widetilde{w}_7$.

    The missing generators are harder to describe: in a nutshell, there is a homomorphism $W(\sfE_7) \to \Sm_{2016} \to \O_{2016}$ induced by the action of $W(\sfE_7)$ on the space of cosets $|U \backslash W(\sfE_7)|$, where $U\subset W(\sfE_7)$ is a certain subgroup. 
    Thus, to a $W(\sfE_7)$-torsor one can associate a quadratic form of rank $2016$. This, regarded in the Witt ring of quadratic forms, turns out to belong to $I^3$, the third power of the fundamental ideal. To this class, scaled by $\langle 2\rangle$, we apply the Milnor-Witt isomorphism $e_3\colon I^{3}/I^{4} \to \K^{\rm Mil}_3 /2$, thus obtaining an invariant $f'_3$. After a further modification, we obtain an invariant $f_3$ which is the lacking generator in degree three, and the product of $f_3$ with the pullback of $\widetilde{w}_1$ gives the last missing generator.
\end{rmk}
\begin{cor}\label{cor:inv Sp6}
    The cohomological invariants of $\sB\Sp_6(2)$ form a free $\M^{\bullet}(\bfk)$-module of rank four, generated by \[ \Inv(\sB\Sp_6(2), \M)\simeq \M^{\bullet}(\bfk)\langle 1, \widetilde{w}_2, f_3,\widetilde{w}_4, \widetilde{w}_6 \rangle. \]
    The invariants $\widetilde{w}_i$ comes from the Galois-Stiefel-Whitney classes of $\sB \O_8$.
\end{cor}
\begin{proof}
    Write $W(\sfE_7)\times\bmu_2$. Then from the formula for the invariants of a product of groups \cite{GMS}*{Part I, Exercise 16.1}, \cite{Hir}*{Proposition 2.5} we have
    \[\Inv(\sB W(\sfE_7)) \simeq \Inv(\sB\Sp_6(2)) \oplus \Inv (\sB\Sp_6(2))\cdot \widetilde{w}_1,\]
    from which our conclusion follows.
\end{proof}
\subsection{Theta characteristics}
In a previous work \cite{DLP2}, Roberto Pirisi and the author considered the moduli stack $\sS_3^{-} \to \sM_3$ of smooth curves of genus three with an odd theta characteristic as a possible source of cohomological invariants for $\sM_3$. Indeed, this moduli stack induces a morphism $\sM_3 \to \sB S_{28}$, obtained by considering the associated $S_{28}$-torsor $\underline{\operatorname{Isom}}_{\sM_3}(\sS_3^{-},\sM_3^{\times 28}) \to \sM_3$. One might wonder whether this map is a source of cohomological invariants distinct from the one of \Cref{thm:inj}. We show that this is not the case.

Let $V$ be the $\ZZ/2$-vector space of rank $2g$ endowed with the standard symplectic structure. Recall from \Cref{sub:symp} that there is a morphism $\Sp_{2g}(2)=\Sp(V) \to S_{2^{g}(2^{g-1}-1)}$ given by $\Sp(V)$ acting on the subset of $QV$ made of odd quadratic forms.
\begin{prop}\label{prop:no new invs}
    The classifying morphism $\sM_g \to \sB S_{2^g(2^{g-1}-1)}$ given by odd theta characteristics factors through the classifying morphism $\sM_g \to \sB \Sp_{2g}(2)$ given by full level structures of order two. In particular, there is a factorization 
    \[\Inv(\sB S_{2^g(2^{g-1}-1)}) \longrightarrow \Inv(\sB\Sp_{2g}(2)) \longrightarrow \Inv(\sM_g). \]
    The same holds for even theta characteristics as well.
\end{prop}
\begin{proof}
    The statement we want to prove is equivalent to saying that, given the cartesian diagram
    \[
    \begin{tikzcd}
        F \ar[r] \ar[d] & \sM_g(2) \ar[d] \\
        \underline{\operatorname{Isom}}_{\sM_g}(\sS_g^{-},\sM_{g}^{2^g(2^{g-1}-1)}) \ar[r] & \sM_g,
    \end{tikzcd}
    \]
    the $S_{2^g(2^{g-1}-1)}$-torsor $F \to \sM_g(2)$ is trivial.
    This blows down to showing that given a full level two structure on a curve $C\to B$, there are $2^g (2^{g-1}-1)$ odd theta characteristics of $C$ defined over $B$. By \Cref{lm:level structures to aronhold}, the choice of a symplectic basis over $B$ of $\underline{\operatorname{Jac}}_{C/B}[2]$ determines an Aronhold basis of $Q(\underline{\operatorname{Jac}}_{C/B}[2])$ defined over $B$, the latter being identified with theta characteristics. Once an Aronhold basis is defined over $B$, all the theta-characteristics are defined over $B$.
\end{proof}
Below we outline how the invariants coming from odd theta characteristics and the ones coming from level two structures are related. Consider again the Weyl group $W(\sfE_7)$, acting on the Picard group of the generic del Pezzo surface $X$. The double cover $\pi\colon X\to \PP^2$ ramified over the generic quartic $C$ can be used to define an action of $W(\sfE_7)$ on the set of bitangents of $C$, with kernel the reflection of length $7$ of $W(\sfE_7)$. This induces a morphism $W(\sfE_7) \to S_{28}$, which factors through the embedding $\Sp_6(2) \to S_{28}$.

Let $\ZZ/2^{\times 7} \subset W(\sfE_7)$ be a maximal abelian subgroup, generated by the reflections $r_1,\ldots, r_7$ such that $\iota=\prod_i r_i$. Equivalently, the reflections $r_i$ are the reflections with respect to hyperplanes $H_i=\alpha_i^{\perp}$, with $\alpha_1,\ldots, \alpha_7$ a set of roots that forms an orthogonal basis of  $\operatorname{Pic}^0(X)$.

Up to conjugation, the image of $\ZZ/2^{\times 7}$ along $W(\sfE_7)\to S_{28}$ is contained in the maximal abelian subgroup $\ZZ/2^{\times 14} \subset S_{28}$ generated by $14$ pairwise disjoint transpositions. We have therefore induced homomorphisms
\[
    \varphi\colon \ZZ/2^{\times 7} \longrightarrow \ZZ/2^{\times 14}, \quad \varphi^*\colon\Inv (\sB \ZZ/2^{\times 14}) \longrightarrow \Inv (\sB \ZZ/2^{\times 7}).
\]
For $W$ a crystallographic Weyl group with maximal abelian subgroup $H$ having normalizer $N$, the restriction homomorphism $\Inv(\sB W) \to \Inv(\sB H)$ is injective \cite{Ser}*{Theorem B}. We have then a commutative square
\[
\begin{tikzcd}
    \Inv(\sB S_{28}) \ar[r] \ar[d] & \Inv( \sB W(\sfE_7)) \ar[d] \\
    \Inv(\sB \ZZ/2^{\times 14})^{S_{14}} \ar[r] & \Inv(\sB \ZZ/2^{\times 7})
\end{tikzcd}
\]
where all but the top arrow are injective. 

The invariants of $\sB\ZZ/2^{\times n}$ form a free $\M^{\bullet}(\bfk)$-module \cite{Hir}*{3.2}. We briefly recall what the generators are: given a multi-quadratic algebra with an isomorphism $E\simeq\bigtimes_{i=1}^{m}K[x_i]/(x_i^2-u_i)$ define the $\K^{\rm Mil}/2$-invariant $\alpha_i(E)=\{u_i\}$. 
A basis for $\Inv(\sB\ZZ/2^{\times n}, \K^{\rm Mil}/2)$ is given by monomials $\alpha_{i_1}\cdot\ldots\cdot \alpha_{i_m}$ for $I=\{i_1,\ldots, i_m\} \subset [n]$.
As $\M^{\bullet}$ is annihilated by $2$, these invariants can be used to define analogous invariants for such cycle modules.
\begin{lm}
    Given a homomorphism $\varphi\colon \ZZ/2^n \to \ZZ/2^m$ represented by an $m\times n$ matrix with entries $c_{ij}$, the homomorphism $\varphi^*$ of cohomological invariants satisfies
\[ \varphi^*(\alpha_i) = \sum_{j=1}^{n} c_{ij}\alpha_j.\]
\end{lm}
\begin{proof}
    Given a multi-quadratic \'{e}tale $K$-algebra $E=\prod_{j=1}^{n} K[x_j]/(x_j^2-u_j)$, the image through the morphism $\sB \ZZ/2^{\times n}(K) \to \sB\ZZ/2^{\times m} (K)$ is the multi-quadratic algebra $E'=\prod_{i=1}^{m} K[y_i]/(y_i^2 - \prod_j u_j^{c_{ij}})$. By definition of the invariants $\alpha_i$, we get
    \[ \varphi^*\alpha_i (E) = \alpha_i (E') = \{\prod_j u_j^{c_{ij}}\} = \sum c_{ij}\alpha_j (E).\]
\end{proof}
The Galois-Stiefel-Whitney classes $\widetilde{w}_d$ of $\sB S_{28}$, restricted to $\sB \ZZ/2^{\times 14}$, are equal to $\sigma_d(\alpha_1,\ldots, \alpha_{14})$, where $\sigma_d$ is the elementary symmetric polynomial of degree $d$. We have proved the following.
\begin{prop}\label{prop:relation}
    Let $(c_{ij})$ be the $14\times 7$ matrix representing the morphism $\ZZ/2^{\times 7} \to \ZZ/2^{\times 14}$ induced by taking the image of the maximal abelian subgroup of $W(\sfE_7)$ in $S_{28}$. Then the homomorphism $\varphi^*\colon \Inv(\sB S_{28}) \to \Inv(\sB W(\sfE_7))$ sends
    \[ \varphi^* \widetilde{w}_d = \sigma_d \left(\sum_{j=1}^{7} c_{1,j}\alpha_j,\ldots,\sum_{j=1}^{7}c_{14,j}\alpha_j \right).\]
\end{prop}
\begin{rmk}
    The element on the right hand side of the equation comes by construction from $\Inv(\sB \Sp_6(2))$, hence it can be expressed in terms of the generators of the latter.
\end{rmk}
\begin{rmk}\label{rmk:rel hyp}
    Jaramillo-Puentes and Pirisi showed that the pullback along $\sM_g \to \sB S_{2^g}$ of certain Galois-Stiefel-Whitney classes is non-zero, at least when the base field is totally real \cite{JPP}. They achieve this by computing the pullback of these classes on hyperelliptic test curves of genus $g$. 
    
    We can reinterpret their approach as follows: the restriction of the torsor of level structures on the substack $\sH_g$ of hyperelliptic curves admits a reduction of group structure; in other terms, the morphism $\sH_g \to \sB \Sp_{2g}(2)$ factors through the classifying morphism $\sH_g \to \sB S_{2g+2}$ associated to the Weierstrass divisor. 
    
    This means that the approach used to obtain \Cref{prop:relation} can be replicated in this setting as well: indeed, by looking at the induced morphism $\varphi\colon \ZZ/2^{\times (g+1)} \to \ZZ^{\times 2^{g-1}} \subset S_{2^g}$ of maximal abelian subgroups,  we can write down the matrix representing $\varphi$, and get a complete description of the pullback homomorphism $$\Inv(\sB S_{2^g}) \longrightarrow \Inv(\sM_g).$$
\end{rmk}
\subsection{Brauer group}
Roberto Pirisi and the author proved that the Brauer group of $\sM_3$ contains a copy of $\ZZ/2$ generated by the pullback of the second Stiefel-Whitney class of $\sB S_{28}$ \cite{DLP2}. This implies that the latter is equal to the pullback of the unique element of the basis of $\Inv(\Sp_6(2))$ of degree two. 

We deduce the following interpretation of this generator: the group $\Sp_6(2)$ has an \'{e}tale, non-trivial double cover given by the metaplectic group ${\rm Mp}_6(2)$. By taking the associated classifying stacks, we have a $\ZZ/2$-gerbe $\sB {\rm Mp}_6(2) \to \sB \Sp_6(2)$. Our main theorem implies the following slight generalization.
\begin{cor}
    The Brauer group of $\sM_{3,n}$ contains a non-trivial $\ZZ/2$-factor, generated by the pullback along $\sM_{3,n} \to \sM_3$ of the gerbe $\sM_3(2)/{\rm Mp}_6(2) \to \sM_3$. 
\end{cor}
\subsection{Mod $2$ cohomology of $\sM_3$ in degree three}
We aim at proving the following.
\begin{prop}\label{thm:degree three}
    Over an algebraically closed field $\bfk$ of characteristic $\neq 2$, the mod $2$ cohomological invariants of $\sM_3$ of degree three come from cohomology. Moreover, we have $\dim \H^3_{\text{\'{e}t}}(\sM_3,\ZZ/2) \geq 2$. 
\end{prop}
We need some preliminary results on Chow groups.
\begin{lm}\label{lm:pullback inj}
    The pullback homomorphism $\operatorname{CH}^2(\sM_3)\otimes \ZZ/2 \longrightarrow \operatorname{CH}^2(\sH_3)\otimes\ZZ/2$ is injective.
\end{lm}
Before proving this lemma, we recall some basic facts about moduli of hyperelliptic curves of genus three. These are mostly taken from \cite{DL-Hyp}.

Let $(\pi\colon C \to S,\iota)$ be a hyperelliptic curve of genus three, with involution $\iota$: then we have a canonical cyclic ramified covering $f\colon C\to P$ of degree two, with $\rho\colon P \to S$ a Severi-Brauer variety of relative dimension one. In particular $f_*\Ocal_C\simeq \Ocal_P\oplus L$, where $f^*{L^{-1}}$ is the ramification divisor. We can define two vector bundles on $\sH_3$, whose Chern classes generate the Chow ring:
\[\cL (\pi) = \rho_*(\omega_\rho^{\otimes 2} \otimes L^{-1}),\quad \cE(\pi)=\rho_*\omega_\rho^{-1}.\]
Set $\tau=c_1(\cL)$ and $c_2=c_2(\cE)$. The following follows from the main theorem of \cite{DL-Hyp}.
\begin{lm}\label{lm:chow hyp}
    We have $\operatorname{CH}^2(\sH_3)\otimes \ZZ/2\simeq \ZZ/2\langle \tau^2,c_2\rangle$
\end{lm}
We compute the restriction of the Hodge bundle of $\sM_3$ over $\sH_3$.
\begin{lm}\label{lm:hodge restriction}
    Let $\mathbb{E}$ be the Hodge bundle over $\sM_3$. Then $\mathbb{E}|_{\sH_3}=\cE\otimes \cL$.
\end{lm}
\begin{proof}
    Let $\pi\colon C\to S$ be a hyperelliptic curve of genus three, and let $f\colon C\to P$ be the associated ramified covering of degree two. Then we have
    \[f_*\omega_{\pi} \simeq \omega_\rho\otimes f_*\Ocal(R)\simeq \omega_\rho \otimes L^{-1} \otimes (\Ocal\oplus L) \simeq \omega_\rho\oplus (\omega_\rho\otimes L^{-1}).\]
    Recall that the Hodge bundle $\mathbb{E}\to \sM_3$ is defined functorially as
    \[ \mathbb{E}(\pi\colon C\to S)= \pi_*\omega_{\pi}.\]
    We deduce that $\pi_*\omega_\pi \simeq \rho_*(\omega_\rho\otimes L^{-1})$.

    By definition of $\cL$, we have $\rho^*(\cL(\pi)) \simeq \omega_\rho^{\otimes 2}\otimes L^{-1}$, hence
    \[ \pi_*\omega_{\pi}=\rho_*(\omega_\rho\otimes L^{-1})=\rho_*(\omega_{\rho}^{-1}\otimes \rho^*\cL)=\rho_*(\omega_\rho^{-1})\otimes \cL. \]
    This proves the formula.
\end{proof}
Next, we compute $\operatorname{CH}^2(\sM_3)\otimes \ZZ/2$. Set $\lambda_i=c_i(\mathbb{E})$.
\begin{lm}\label{lm:chow m3}
    We have $\operatorname{CH}^2(\sM_3)\otimes\ZZ/2\simeq \ZZ/2\langle \lambda_1^2,\lambda_2\rangle$
\end{lm}
\begin{proof}
    The main theorem of \cite{DLFV} tells us that $$\operatorname{CH}^2(\sM_3\smallsetminus\sH_3)\otimes \ZZ/2\simeq \ZZ/2\langle \lambda_2 \rangle.$$
    The localization exact sequence reads 
    \[ \operatorname{CH}^1(\sH_3)\otimes\ZZ/2\overset{\iota}{\longrightarrow }\operatorname{CH}^2(\sM_3)\otimes\ZZ/2 \longrightarrow \operatorname{CH}^2(\sU_3)\otimes\ZZ/2 \longrightarrow 0 \]
    From \cite{DL-Hyp} it follows that $\operatorname{CH}^1(\sH_3)\otimes\ZZ/2\simeq \ZZ/2\cdot\tau$.
     We claim that $\iota$ is injective, which would conclude the proof. 
     
     First, observe that $\iota^*\lambda_1=c_1(\mathbb{E})=\tau$, hence $\iota_*\tau=\iota_*\iota^*\lambda_1=[\sH_3]\cdot \lambda_1$. As $[\sH_3]=9\lambda_1=\lambda_1$, we deduce that $\iota$ is injective if and only if $\lambda_1^2$ is not zero in $\operatorname{CH}^2(\sM_3)\otimes\ZZ/2$. But $\iota^*\lambda_1^2=\tau_1^2\neq 0$, hence the claim is proved.
\end{proof}
\begin{proof}[Proof of \Cref{lm:pullback inj}]
    From \Cref{lm:chow m3}, \Cref{lm:chow hyp} and \Cref{lm:hodge restriction} we deduce that the pullback homomorphism $\iota$, regarded as a morphism of $\ZZ/2$-vector spaces, is
    \[ \iota:\ZZ/2\langle \lambda_1^2,\lambda_2\rangle \longrightarrow \ZZ/2\langle \tau^2,c_2\rangle, \quad \iota^*\lambda_1^2=\tau^2,\quad \iota^*\lambda_2=\tau^2+c_2. \]
    Injectivity of this map is clear.
\end{proof}
\begin{proof}[Proof of \Cref{thm:degree three}]
Let $S_n$ be the symmetric group of order $n$: the cohomological invariants of $\sB S_n$ come from cohomology. The Weierstrass divisor of the universal hyperelliptic curve over $\sH_3$ induces a classifying map $\sH_3 \to \sB S_8$, and the induced pullback homomorphism $\operatorname{Inv}^3(\sB S_8, \ZZ/2) \to \operatorname{Inv}^3(\sH_3, \ZZ/2)$ is an isomorphism \cite{DLP}*{Main Theorem B}. By \Cref{cor:hom triv} we deduce that $\operatorname{CH}^2(\sH_3)\otimes\ZZ/2 \to \H_{\text{\'{e}t}}^4(\sH_3,\ZZ/2) $ is injective, hence by \Cref{lm:pullback inj} the composition
\[ \operatorname{CH}^2(\sM_3)\otimes \ZZ/2 \overset{\iota^*}{\longrightarrow }\operatorname{CH}^2(\sH_3)\otimes\ZZ/2 \overset{\operatorname{cl}}{\longrightarrow}\H_{\text{\'{e}t}}^4(\sH_3,\ZZ/2) \]
is injective as well.
As $\operatorname{cl}_{\sH_3}\circ \iota^* = \iota^*\circ \operatorname{cl}_{\sM_3}$, we deduce that $\operatorname{cl}_{\sM_3}$ is injective for codimension two cycles. By \Cref{cor:hom triv}, we deduce that the degree three invariants of $\sM_3$ come from cohomology.

Next, we prove that $\dim \H^3_{\text{\'{e}t}}(\sM_3,\ZZ/2)\geq 2$.
    The pullback homomorphism $$\H^2_{\text{\'{e}t}}(\sM_3,\ZZ/2) \simeq \operatorname{Br}(\sM_3)_2 \longrightarrow \operatorname{Br}(\sU_3)_2\simeq \H^2_{\text{\'{e}t}}(\sU_3,\ZZ/2)$$ is an isomorphism \cite{DLP2}*{Proposition 3.8}, hence from the long exact sequence in \'{e}tale cohomology we deduce the exactness of 
    \[ 0 \longrightarrow \H_{\text{\'{e}t}}^1(\sH_3,\ZZ/2) \overset{\iota_*}{\longrightarrow} \H_{\text{\'{e}t}}^3(\sM_3,\ZZ/2) \overset{j^*}{\longrightarrow} \H_{\text{\'{e}t}}^3(\sU_3,\ZZ/2). \]
    The first term coincides with the $2$-torsion in the Picard group of $\sH_3$, hence it is isomorphic to $\ZZ/2\cdot \gamma$. On the other hand, we have a commutative diagram
    \[ 
    \begin{tikzcd}
        \H_{\text{\'{e}t}}^3(\sM_3,\ZZ/2) \ar[r,"j^*"] \ar[d] & \H_{\text{\'{e}t}}^3(\sU_3,\ZZ/2) \ar[d] \\
        \operatorname{Inv}^3(\sM_3,\H_{\ZZ/2}) \ar[r, "\simeq"] & \operatorname{Inv}^3(\sU_3,\H_{\ZZ/2}).
    \end{tikzcd}
    \]
    As $j^*\iota_*(\gamma)=0$, we deduce that $\iota_*\alpha$ is sent to zero in $\operatorname{Inv}^3(\sM_3,\H_{\ZZ/2})$. As the latter has dimension $\geq 1$, we deduce that $\dim \H_{\text{\'{e}t}}^3(\sM_3,\ZZ/2) \geq 2$.
\end{proof}

\bibliographystyle{halpha-abbrv}
\bibliography{bibliography.bib}

\end{document}